\documentclass[flen,a4paper]{article}
\usepackage{amsmath,amsthm,amssymb,amscd}
\usepackage{graphicx}
\usepackage{graphics}
\usepackage{verbatim}
\usepackage{enumerate}
\usepackage{mathrsfs}
\usepackage{color}
\usepackage[top=1in, bottom=0.8in, left=0.8in, right=0.8in]{geometry}

\makeatletter

\newcommand{\Rmnum}[1]{\expandafter\@slowromancap\romannumeral #1@}

\usepackage{color}

\newtheorem{theorem}{Theorem}[section]

\newtheorem{problem}[theorem]{Problem}
\newtheorem{lemma}[theorem]{Lemma}
\newtheorem{definition}[theorem]{Definition}
\newtheorem{corollary}[theorem]{Corollary}

\makeatother

\newcommand{\intt}{{\rm int}}
\newcommand{\ext}{{\rm ext}}

\begin{document}
	
	\title{$(I,F)$-partition of planar graphs without cycles of length 4, 6, or 9}

	\vspace{3cm}
	\author{Yingli Kang\footnotemark[1]~, Hongkai Lu\footnotemark[2]~, Ligang Jin\footnotemark[3]}
	
	\footnotetext[1]{Department of Mathematics, Jinhua Polytechnic, Western Haitang Road 888, 321017 Jinhua, China. Email: ylk8mandy@126.com}
		
	\footnotetext[2]{College of Mathematics,
	Zhejiang Normal University, Yingbin Road 688,
	321004 Jinhua,
	China. Email: 1298488241@qq.com}

	\footnotetext[3]{Corresponding author. College of Mathematics,
	Zhejiang Normal University, Yingbin Road 688,
	321004 Jinhua,
	China.
	Email: ligang.jin@zjnu.cn}
	
	\date{}
	
	\maketitle

\begin{abstract}
  A  graph $G$ is $(I,F)$-partitionable if its vertex set can be partitioned into two parts such that one part
  is an independent set, and the other induces a forest. 
  In this paper, we prove that every planar graph without cycles of length  $4, 6, 9$ is $(I,F)$-partitionable.

\end{abstract}


\section{Introduction}
The planar graph $G$ considered in this paper is finite and simple.
A graph $G$ is \emph{$k$-degenerate} if every subgraph $H$ of $G$ contains a vertex of degree at most $k$ in $H$.
Clearly, $k$-degenerate graphs are $(k+1)$-colorable. Let $p$ and $q$ be two nonnegative integers.
A graph $G$ is \emph{$(p,q)$-partitionable} if $V(G)$ can be partitioned into two subsets which induces a $p$-degenerate subgraph and a $q$-degenerate subgraph of $G$, respectively.
Thomassen \cite{C1,C2} proved that planar graphs are both $(1,2)$-partitionable and $(0,3)$-partitionable.

A  graph $G$ is \emph{$(I,F)$-partitionable} if its vertices set can be partitioned into two parts, one part
is an independent set, and the other induces a forest. 
Clearly, $(I,F)$-partitionable is same as $(0,1)$-partitionable.
Borodin and Glebov \cite{O.A 2001} confirmed that every planar graph of girth at least 5 is $(I,F)$-partitionable. Kawarabayashi and  Thomassen \cite{K.C} proved an extension of this result and guessed it might be true that every triangle-free planar graph is $(I,F)$-partitionable. 

Clearly, every $(I,F)$-partitionable graph is signed 3-colorable and is further 3-colorable. Liu and Yu \cite{R.G np468} proved that planar graph without cycles of length $4$, $6$, or $8$ are $(I,F)$-partitionable, which extends the result of Wang and Chen \cite{W.C} that they are 3-colorable.
In this paper, we are interested in the following problem with the same flavour. 

\begin{problem}\label{pro_4ij_IF}
	For which pair of integers $(i,j)$ with $4< i< j< 10$, planar graphs without cycles of length from $\{4,i,j\}$ are $(I,F)$-partitionable.
\end{problem}

Denote by $G[S]$ the subgraph of a graph $G$ induced by a set $S$ with $S\subseteq V(G)$ or $S\subseteq E(G)$.
\begin{definition}
	Let $C$ be a cycle of a plane graph $G$. 
	An edge inside $C$ connecting two non-consecutive vertices of $C$ is called a \emph{chord} of $C$. 
	If a vertex $v\in \intt(C)$ has three neighbors $v_1,v_2,v_3$ on $C$, then $G[\{vv_1,vv_2, vv_3\}]$ is called a \emph{claw} of $C$.
	If $u\in \intt(C)$ has two neighbors $u_1$ and $u_2$ on $C$, $v\in \intt(C)$ has two neighbors $v_1$ and $v_2$ on $C$, and $uv\in E(G)$, then $G[\{uv,uu_1,uu_2,vv_1,vv_2\}]$ is called a \emph{biclaw} of $C$.
	If each of three pairwise adjacent vertices $u,v,w\in \intt(C)$ has a neighbor on $C$, say $u', v', w'$ respectively, then
	$G[\{uv,vw,uw,uu',vv',ww'\}]$ is called a \emph{triclaw} of $C$.
	The cycles into which a chord, a claw, a biclaw, or a triclaw  divides $C$ are called \emph{cells}, see Figure \ref{fig-claw}.
	A cell of length $c_i$ is called a \emph{$c_i$-cell}.
	We further call a $(c_1,c_2)$-chord, a \emph{$(c_1,c_2,c_3)$-claw}, a \emph{$(c_1,c_2,c_3,c_4)$-biclaw}, or a \emph{$(c_1,c_2,c_3,c_4)$-triclaw}, as depicted in Figure \ref{fig-claw}.
	\begin{figure}[ht]
		\centering
		\includegraphics[width=14cm]{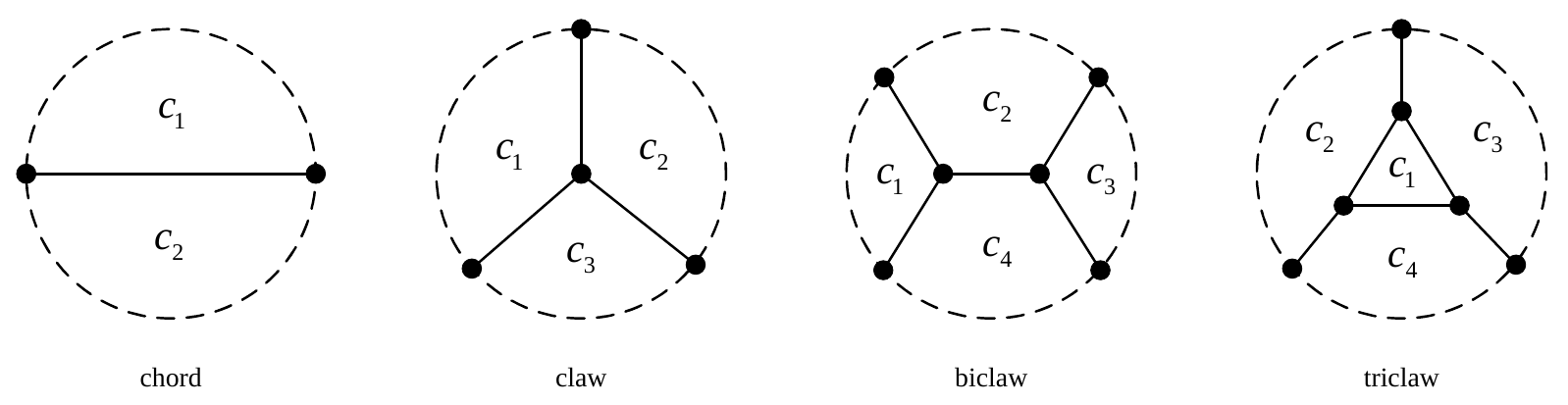}
		\caption{A cycle $C$ in dotted line and a chord, a claw, a biclaw, and a triclaw of $C$ in solid line}\label{fig-claw}
	\end{figure}

\end{definition}
A 9-cycle is \emph{special} if it has a $(3,8)$-chord or a $(5,5,5)$-claw.
Let $\mathcal{G}$ denote the class of connected  plane graphs with neither 4- or 6-cycles nor special 9-cycles.
The following theorem is the main result of this paper.
\begin{theorem} \label{thm469}
	Every graph of $\mathcal{G}$ is $(I,F)$-partitionable.
\end{theorem}

This theorem implies the result of Liu and Yu \cite{R.G np468} and a new result on Problem \ref{pro_4ij_IF} as follows.

\begin{corollary}[\cite{R.G np468}] \label{cor468}
	Every planar graph without cycles of length 4, 6, or 8 is $(I,F)$-partitionable.
\end{corollary}

\begin{corollary} \label{cor469}
	Every planar graph without cycles of length 4, 6, or 9 is $(I,F)$-partitionable.
\end{corollary}

\begin{corollary} \label{cor469s}
	Every planar graph without cycles of length 4, 6, or 9 is signed 3-colorable.
\end{corollary}

 \section{Super-extended theorem}
An \emph{$(I,F)$-coloring} of a graph $G$ is a mapping from $V(G)$ to the color set $\{I,F\}$ such that vertices of the color $I$ is an independent set and vertices of the color $F$ induce a forest.
A vertex of color $F$ is called an \emph{$F$-vertex}.
A path or cycle on only $F$-vertices is called an \emph{$F$-path} or \emph{$F$-cycle}, respectively.
Let $H$ be a subgraph of a graph $G$ and $\phi$ an $(I,F)$-coloring of $H$.
A \emph{super-extension} of $\phi$ to $G$ is an $(I,F)$-coloring of $G$ whose restriction on $H$ is $\phi$ such that $G-E(H)$ contains no $F$-path connecting two vertices of $H$. 
A cycle of length at most 12 is \emph{bad} if it has a claw, a biclaw or a triclaw; \emph{good} otherwise.

We will prove the following theorem, which strengthens Theorem \ref{thm469}.
\begin{theorem} (Super-extended theorem)  \label{thm_main_extension}
	Let $G\in \cal{G}$. If the boundary $D$ of the unbounded face of $G$ is a good cycle, then every $(I,F)$-coloring of $G[V(D)]$ can super-extend to $G$.
\end{theorem}

To see that Theorem \ref{thm469} follows from Theorem \ref{thm_main_extension}, take any graph $G\in \mathcal{G}$. If $G$ has no triangles, then it has girth at least 5 and is known to be $(I,F)$-partitionable \cite{O.A 2001}. So, let $T$ be a triangle of $G$.
If there are $10^-$-cycles containing $T$ inside, then let $C$ be the outermost one, that is, the one which is contained in the interior of no other $10^-$-cycles; otherwise, let $C=T$. Take any $(I,F)$-coloring $\phi$ of $C$. Then $\phi$ can super-extend to both $\ext[C]$ and $\intt[C]$ by Theorem \ref{thm_main_extension}. This results in an $(I,F)$-coloring of $G$.

Given a plane graph $G$. A path is a \emph{splitting path} of a cycle $C$ if its two end-vertices locate on $C$ and all other vertices locate inside $C$. A path or a cycle $C$ is \emph{triangular} if it has an edge as the common part between $C$ and some triangle.
A cycle $C$ is \emph{ext-triangular} if it has an edge as the common part between $C$ and some triangle of $\ext[C]$.
Given a $(I,F)$-coloring of $G$. A pair of vertices $(u,v)$ is \emph{$F$-linked} if at least one of the following holds:
\begin{enumerate}[(1)]
    \setlength{\itemsep}{0pt} 
	\item there exists an $F$-path between $u$ and $v$;
	\item there exist two vertex-disjoint $F$-paths, one connects $u$ and some external vertex, and the other connects $v$ and another external vertex.
\end{enumerate}

Given an $(I,F)$-coloring of a subgraph $H$ of a plane graph $G$. \emph{Nicely color} an uncolored vertex $u$ assigns with the color $I$ if $u$ has no neighbors of color $I$ and assigns with the color $F$ otherwise. 
It is easy to see that if $u$ has at most two neighbors in $H$, then nicely coloring $u$ results in an $(I,F)$-coloring and brings neither $F$-cycle nor splitting $F$-path of $D$, that is, $u$ is contained in neither $F$-cycle nor splitting $F$-path of $D$.
Let $\vec{P}=v_1v_2\ldots v_k$ be a path of $G-V(H)$ with direction along this path starting from $v_1$ such that the neighbors of $v_i$ in $H$ are of the same color (say $\alpha_i$) for each $i\in \{1,2,\ldots,k\}$. 
\emph{$I$-nicely color} (resp. \emph{$F$-nicely color}) $\vec{P}$ assigns $v_i$ with the color $F$ for each $i$ with $\alpha_i=I$ and then assigns all the remaining vertices of $P$ with first $I$ and then $F$ (resp. with first $F$ and then $I$) alternately.
It is easy to see that both $F$-nicely coloring $\vec{P}$ and $I$-nicely coloring $\vec{P}$ result in an $(I,F)$-coloring and bring neither $F$-cycle nor splitting $F$-path of $D$.

\section{The proof of Theorem \ref{thm_main_extension}}    \label{sec_proof}

We shall prove Theorem \ref{thm_main_extension} by contradiction.
Let $G$ be a counterexample to Theorem \ref{thm_main_extension} with minimum $|V(G)|+|E(G)|$.
Thus, the boundary $D$ of the unbounded face $f_0$ of $G$ is a good cycle, and there exists an $(I,F)$-coloring $\phi_0$ of $G[V(D)]$ which can not super-extend to $G$.

\subsection{Reducible configurations}
\begin{lemma} \label{lem_min degree}
Every internal vertex of $G$ has degree at least 3.
\end{lemma}

\begin{proof}\label{lem_min_degree}
 Otherwise, let $d(v)\leq 2$. The pre-coloring $\phi_0$ can super-extend to $G-v$ by the minimality of $G$, and further to $G$ by nicely coloring $v$.
\end{proof}

\begin{lemma} \label{lem_separating-cycle}
	$G$ has no separating good cycle.
\end{lemma}

\begin{proof}
    If $C$ is a separating good cycle of $G$, then $\phi_0$ can super-extend to $G-\intt(C)$ by the minimality of $G$, and the resulting coloring of $C$ can super-extend to $\intt[C]$. This results in a super-extension of $\phi_0$ to $G$, a contradiction.
\end{proof}

The following three lemmas can be concluded easily.
\begin{lemma}\label{facial cycle}
	Every $9^-$-cycle of $G$ is facial except that an 8-cycle of $G$ might have a $(3,7)$- or $(5,5)$-chord.
\end{lemma}

\begin{lemma}\label{bad cycle}
	Let $H\in \cal{G}$. If $C$ is a bad cycle of $H$, then $C$ has length either 11 or 12. Furthermore, if $|C|=11$, then $C$ has a $(3,7,7)$- or $(5,5,7)$-claw; if $|C|=12$, then $C$ has a $(5,5,8)$-claw, a $(3,7,5,7)$- or $(5,5,5,7)$-biclaw, or a $(3,7,7,7)$-triclaw.
\end{lemma}

\begin{lemma}\label{bad cycle 1}
	Every bad cycle $C$ of $G$ is adjacent to at most one triangle. Furthermore, if  $C$ is ext-triangular, then  $C$ has a $(5,5,7)$-claw or  $(5,5,5,7)$-biclaw.
\end{lemma}

\begin{lemma}\label{lem_2connected}
	$G$ is 2-connected.
\end{lemma}

\begin{proof}
	Otherwise, we may assume that $G$ has a block $B$ and a cut vertex $v\in V(B)$.
	By the minimality of $G$, $\phi_0$ can super-extend to $G-V(B-v)$.
	Consider only $B$. If $v$ is contained in $10^-$-cycles, then take the outermost one, that is, the one which is contained in the interior of no other $10^-$-cycles, denoted by $C$.
	Lemma \ref{bad cycle} implies that $C$ is good and therefore, the coloring of $v$ can extend to an $(I,F)$-coloring of $C$ which can further super-extend to both the interior and exterior (if not empty) of $C$ in $B$ by the minimality of $G$. This results in an $(I,F)$-coloring of $B$.
	It remains to assume that $v$ is contained in no $10^-$-cycles.
	Insert into the unbounded face $f$ of $B$ an edge $e$ between the two neighbors of $v$ on $f$, creating a 3-face, say $T$. Note that the embedding of $B+e$ in the plane which takes $T$ as the unbouded face belongs to $\mathcal{G}$. Similarly, the coloring of $v$ can extend to an $(I,F)$-coloring of $T$ and can further super-extend to $B+e$.
	In either case, the resulting coloring of $G$ is a super-extension of $\phi_0$, a contradiction. 
\end{proof}

\begin{lemma}\label{splitting path}
    Let $P$ be a splitting path of $D$, which divides $D$ into two cycles $D'$ and $D''$.
    If $2\leq |P| \leq 5$, then at least one of $D'$ and $D''$ has length $|P|+1$ to $2|P|-1$. 
    More precisely, since $G\in \mathcal{G}$,
    \begin{enumerate}[(1)]
    	\setlength{\itemsep}{0pt} 
    	\item if $|P|=2$, then at least one of $D'$ and $D''$ is a triangle;
    	\item if $|P|=3$, then at least one of $D'$ and $D''$ is a 5-cycle;
    	\item if $|P|=4$, then at least one of $D'$ and $D''$ is a 5- or 7-cycle;
    	\item if $|P|=5$, then at least one of $D'$ and $D''$ is a 7-, 8-, or 9-cycle.
    \end{enumerate}
\end{lemma}

\begin{proof}
	Suppose to the contrary that $|D'|, |D''| \geq 2|P|$.
	Since $D$ has length at most 12, $|D'|+|D''|=|D|+2|P|\leq 12+2|P|$.
	It follows that $2|P|\leq |D'|, |D''| \leq 12.$
	
	(1) Let $P=xyz$. By Lemma \ref{lem_min degree}, $y$ has a neighbor $y'$ other than $x$ and $z$.
	If $y'$ is external, then $D$ has a claw, a contradiction.
	So, $y'$ lies inside $D'$ or $D''$, w.l.o.g., say $D'$. By Lemma \ref{lem_separating-cycle}, $D'$ is a bad cycle. Moreover, since $G$ has no 4-cycles, $5\leq |D'|, |D''| \leq 11.$ Hence by Lemma \ref{bad cycle}, $D'$ has a claw, which yields that $D$ has a biclaw, a contradiction.
	
	(2) Let $P=wxyz$. We may let $x'$ and $y'$ be neighbors of $x$ and $y$ with $\{xx',yy'\}\cap E(P)=\emptyset$, respectively.
    If both $x'$ and $y'$ are external, then $D$ has a biclaw, a contradiction.
    So, w.l.o.g., let $x'$ lie inside $D'$.  
    Moreover, since $G$ has no 6-cycles, $7\leq |D'|, |D''| \leq 11.$
    Hence by Lemmas \ref{lem_separating-cycle} and \ref{bad cycle}, $D'$ is a bad 11-cycle with a claw and $D''$ is a 7-face.
    So, $y'$ has no choices but coincides with $x'$. Now, $D$ has a triclaw, a contradiction.
	
	(3) Let $P=vwxyz$. In this case, $8\leq |D'|, |D''| \leq 12$.
	We claim that $G$ has no edge connecting two non-consecutive vertices on $P$. Otherwise, such an edge $e$ together with $P$ forms a triangle as well as a splitting 3-path of $D$. By the statement (2), we can deduce that $e$ is a (3,5)-chord of $D'$, a contradiction.
	
	Let $w'$, $x'$, and $y'$ be neighbors of $w$, $x$, and $y$ with $\{ww',xx',yy'\}\cap E(P)=\emptyset$, respectively.	
	Clearly, $x'$ lies in $\intt[D']$ or $\intt[D'']$, w.l.o.g., say $\intt[D']$.
	If $x'$ is external, then both the paths $vwxx'$ and $x'xyz$ are splitting 3-paths of $D$. By the statement (2), $D'$ is an 8-cycle with a (5,5)-chord $xx'$.
	Hence, $y'$ has no choice for its location but to lie inside $D''$, and so does $w'$. So, $D''$ is a bad cycle and by Lemma \ref{bad cycle}, either $w'=y'$ which yields a 4-cycle or $w'y'\in E(G)$ which yields a special 9-cycle with a $(5,5,5)$-claw, a contradiction. 
	It remains to assume that $x'\in \intt(D')$.
	Thus, $D'$ is a bad cycle, which implies that $D''$ has length 8 or 9.
	For $|D''|=9$, $D''$ is facial and $D'$ is a bad 11-cycle with a claw, which is impossible because of the locations of $w'$, $x'$ and $y'$. For $|D''|=8$, at least one of $w'$ and $y'$ lies in $\intt[D']$, which together with $x'$ yields either a 4-cycle or a special 9-cycle with a $(3,8)$-chord, a contradiction.
	
	(4) ~Let $P=uvwxyz$. In this case, $10\leq |D'|, |D''| \leq 12$.
	By a similar argument as in the case (3), one can conclude that $G$ has no edge connecting two nonconsecutive vertices on $P$.
	Let $v', w', x', y'$ be neighbors of $v, w, x, y$ not on $P$, respectively.
	
	We claim that both $w'$ and $x'$ are internal. Otherwise, let $w'\in V(D')$. Since both $uvww'$ and $w'wxyz$ are splitting paths of $D$, $D'$ is a 10-cycle with a (5,7)-chord $ww'$. If $x'\in V(D'')$, then similarly, $D''$ is a 10-cycle with a (5,7)-chord $xx'$, which yields no locations for $v'$ and $y'$. Hence, $x'\in \intt(D'')$. Moreover, $v'\in \intt(D'')$ since otherwise, $uvv'$ is a splitting 2-path of $D$ which yields a triangle adjacent to a 5-cycle.  Therefore, $v'x'\in E(G)$ and $D''$ is a bad 12-cycle with a biclaw, which yields no location for $y'$.
	
	If $w'$ and $x'$ lie inside different one between $D'$ and $D''$, then both $D'$ and $D''$ are bad 11-cycles with a claw, yielding $v'=w'$ and $y'=x'$. Now, $G$ has a special 9-cycle with a (3,8)-chord.
	Otherwise, let $w',x'\in \intt(D')$. Since $G$ has no 4-cycles, $x'=w'$ and hence, $D'$ is a bad cycle with either a $(3,7,7)$-claw or a (3,7,5,7)-biclaw.
	If $v'\in V(D'')$, then $uvv'$ is a splitting 2-path of $D$, forming a (3,8)-chord $uv$.
	Hence, $v'\in \intt(D'')$ and similarly, $y'\in \intt(D'')$.
	It follows that either $v'= y'$ or $v'y'\in E(G)$, yielding a 6-cycle in both cases.
\end{proof}

\begin{lemma} \label{operation}
	Let $G'$ be a plane graph obtained from $G$ by deleting a nonempty set of internal vertices and either identifying two vertices without identifying edges or
	adding an edge.
	If we 
	\begin{enumerate} [(a)]
		\setlength{\itemsep}{0pt} 
		\item identify no two vertices on $D$, and create no edge connecting two vertices on $D$, and
		\item create neither $6^-$-cycle nor ext-triangular 7- or 8-cycle,
	\end{enumerate}
	then $\phi_0$ can super-extend to $G'$.
\end{lemma}

\begin{proof}
The item $(a)$ guarantees that $D$ is unchanged and bounds $G'$ and that $\phi_0$ is an $(I,F)$-coloring of $G'$.
By the item $(b)$, $G'$ is simple and $G'$ contains no 4- or 6-cycles.
Hence, to super-extend $\phi_0$ to $G'$ by the minimality of $G$, it suffices to show both that $D$ is a good cycle in $G'$ and that $G'$ contains no special 9-cycles.

Suppose to the contrary that $D$ is a bad cycle of $G'$, i.e., $D$ has a claw, biclaw, or triclaw, say $H$. 
For the case of identifying two vertices, the resulting vertex is incident with $k$ ($k\leq 2$) cells of $H$ that are created by the operation. If $k= 0$, then $D$ has $H$ also in $G$, a contradiction. Moreover, since the operation does not identify edges, $k\neq 1.$ Therefore, $k= 2$.  
It follows by Lemma \ref{bad cycle} that there is a $5^-$-cycle or an ext-triangular 7-cycle created, contradicting the item $(b)$. 
For the case of inserting a new edge, say $e$, we can similarly deduce that both cells of $H$ incident with $e$ are created, yielding a similar contradiction as above.

Suppose to the contrary that $G'$ contains a special 9-cycle $C$.  
By a similar argument on $C$ as on $D$ above, we can deduce that there is a $5^-$-cycle or an ext-triangular 8-cycle created, contradicting the item $(b)$. 	
\end{proof}

\begin{lemma}\label{operation_merge_edge}
	Let $G'$ be a plane graph obtained from $G$ by the following operation $T$: deleting a nonempty set $S$ of internal vertices and identifying two edges $u_1u_2$ and $v_1v_2$ so that $u_1$ is identified with $v_1$.
	For $i\in \{1,2\}$, let $T_i$ denote the operation on $G$ that consists of deleting all the vertices of $S$ and identifying $u_i$ and $v_i$.
	If at least one of $u_1u_2$ and $v_1v_2$ is contained in no $8^-$-cycle of $G-S$, and the conditions $(a)$ and $(b)$ of Lemma \ref{operation} hold for both $T_1$ and $T_2$,
	then $\phi_0$ can super-extend to $G'$.
\end{lemma}

\begin{proof}
	For $i\in\{1,2\}$, denote by $w_i$ the vertex resulting from $u_i$ and $v_i$ by $T$.
	Since the condition $(a)$ holds for both $T_1$ and $T_2$, $D$ bounds $G'$ and $\phi_0$ is an $(I,F)$-coloring of $G'[V(D)]$.
	
	Suppose that $T$ creates a $6^-$-cycle or a special 9-cycle or a bad $D$, denoted by $C$.
	Since the two conditions $(a)$ and $(b)$ hold for both $T_1$ and $T_2$, by the proof of Lemma \ref{operation}, each $T_i$ does not create $C$. Hence, $w_1w_2$ must be either a common edge of some two cells of $C$ or a chord of some cell of $C$. 
	This implies that both $u_1u_2$ and $v_1v_2$ are contained in a $8^-$-cycle of $G-S$, contradicting the assumption.
		
    Therefore, $\phi_0$ can super-extend to $G'$ by the minimality of $G$.
\end{proof}

Given a plane graph. A \emph{good path} is a path $P=v_1v_2v_3v_4$ of the boundary of some face such that the edge $v_1v_2$ is triangular and all the vertices of $P$ are internal 3-vertices, see Figure \ref{fig-good path}. 

\begin{figure}[ht]
	\centering
	\includegraphics[width=7cm]{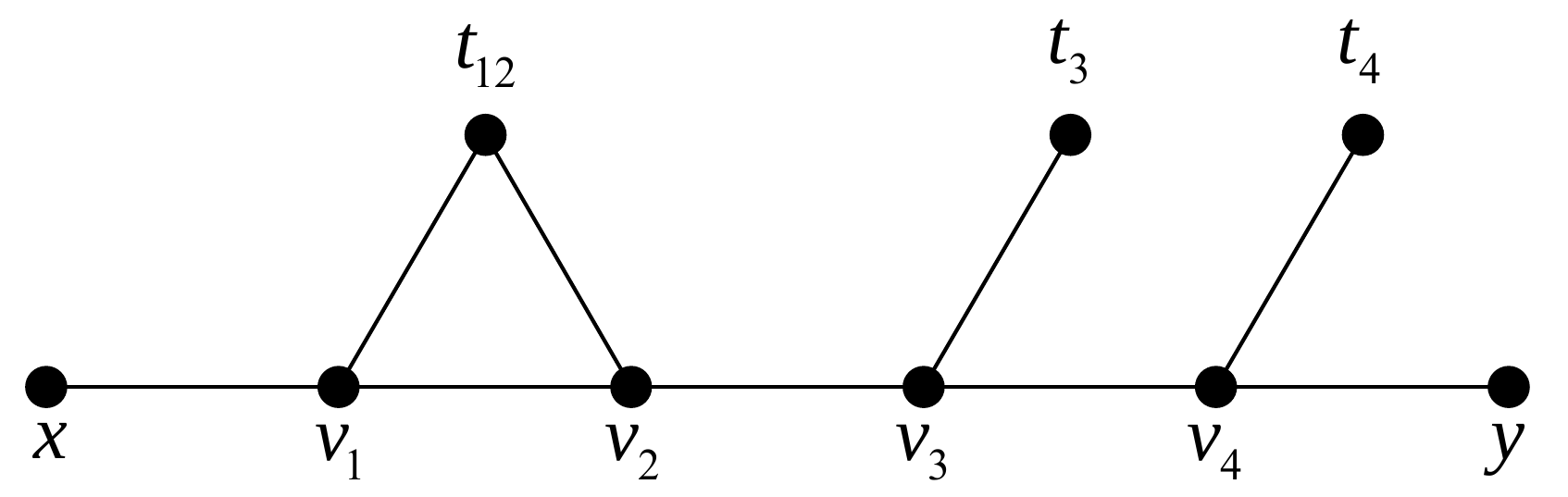}
	\caption{good path}\label{fig-good path}
\end{figure}

\begin{lemma} \label{lem_good path}
	$G$ has no good paths. 
\end{lemma}

\begin{proof}
    Suppose to the contrary that $G$ has a good path $P=v_1v_2v_3v_4$, using the same label for vertices as in Figure \ref{fig-good path}. 
    Since $G\in \mathcal{G}$, all the vertices in Figure \ref{fig-good path} are pairwise distinct except that $t_3$ and $t_4$ might coincide. Apply on $G$ the following operation $T$: remove all the vertice of $P$ and identify $x$ with $t_3$, obtaining a smaller plane graph $G'$.

    Suppose that $T$ creates a $6^-$-cycle or an ext-triangular 7- or 8-cycle. Thus, $G-v_4$ has a $12^-$-cycle $C$ containing $xv_1v_2v_3t_3$ and additionally, if $|C|\in \{11,12\}$ then the path $C-\{v_1,v_2,v_3\}$ is triangular.
    By planarity, $t_{12}\in V(C)$ or $t_{12}\in \intt(C)$ or $v_4\in \intt(C)$. For the first case, between the two cycles formed by paths $C-v_1v_2$ and $v_1t_{12}v_2$, at least one is a triangular $6^-$-cycle, contradicting that $G\in \mathcal{G}$.
    For the last two cases, $C$ is a bad cycle by Lemma \ref{lem_separating-cycle}. But now $C$ is adjacent to two triangles, contradicting Lemma \ref{bad cycle 1}.
    So, the item (b) of Lemma \ref{operation} holds for $T$.

    Suppose that $T$ identifies two external vertices or create an edge connecting two external vertices.
    Thus, $xv_1v_2v_3t_3$ is contained in a splitting 4- or 5-path of $D$, which together with $D$ forms a $9^-$-cycle by Lemma \ref{splitting path}. Thus, $T$ creates a $5^-$-cycle, a contradiction.  
    Therefore, the item $(a)$ of Lemma \ref{operation} holds for $T$.
    
    Hence, $\phi_0$ can super-extend to $G'$ by Lemma \ref{operation} and further to $G$ as follows. 
    Nicely color $v_4$ and $v_3$ in turn.  Clearly, $x$ and $t_3$ receive the same color, say $\alpha$.  Denote by $\beta$ and $\gamma$ the colors of $t_{12}$ and $v_3$, respectively. We distinguish the following four cases.
    
    (\romannumeral1) If $\alpha=I$, then color $v_1$ by $F$ and color $v_2$ different from $t_{12}$. Note that the coloring of $P$ brings neither $F$-cycle nor splitting $F$-path of $D$, we are done.
    
    (\romannumeral2) If $\alpha=F$ and $\beta=I$, then color both $v_1$ and $v_2$ by $F$, we are done. Notice that $xv_1v_2v_3t_3$ might be an $F$-path, which however does not bring an $F$-cycle or a splitting $F$-path of $D$ since otherwise, identifying $x$ with $t_3$ yields an $F$-cycle or a splitting $F$-path of $D$ in $G'$.

    (\romannumeral3) Let $\alpha=\beta=F$ and $\gamma=I$. Color $v_1$ by $I$ and $v_2$ by $F$, we are done.

    (\romannumeral4) Let $\alpha=\beta=F$ and $\gamma=F$.    
    Since identifying $x$ with $t_3$ yields neither $F$-cycle nor splitting $F$-path of $D$ in $G'$, at least one of $(x,t_{12})$ and $(t_{12},t_3)$ is not $F$-linked, for which we color $v_1$ by $F$ or $I$ respectively, and color $v_2$ different from $v_1$. 
\end{proof}

\begin{lemma} \label{lem_5,7-face}
	For $k\in\{5,7\}$, $G$ has no $k$-face that contains $k$ internal 3-vertices.
\end{lemma}

\begin{proof}
Suppose to the contrary that $G$ has such a $k$-face $f=[v_1...v_k]$. Let $t_i$ be the remaining neighbor of $v_i$ for $i\in\{1,2,\dots,k\}$. Since $G\in \mathcal{G}$ and Lemma \ref{lem_good path}, these vertices $t_1,\ldots,t_k$ are pairwise distinct.

   {\bf  Case 1}:
    Let $k=5$. Since $G\in \mathcal{G}$, $f$ contains a vertex incident with two $7^+$-faces, w.l.o.g., say $v_2$.
    Apply on $G$ the following operation $T$: remove $V(f)$ and insert an edge between $t_1$ and $t_3$, obtaining a smaller plane graph $G'$.

    Suppose that $T$ creates a $6^-$-cycle or an ext-triangle 7- or 8-cycle. Then $G-\{v_4,v_5\}$ has an $11^-$-cycle $C$ containing the path $P=t_1v_1v_2v_3t_3$ and additionally, $\ext[C]$ has a triangle sharing an edge with $C-E(P)$ when $|C|\in \{10,11\}$.
    If $C$ is a good cycle, then $t_2\in V(C)$ and thus, $v_2t_2$ is a $(7^+,7^+)$-chord of a $11^-$-cycle $C$, a contradiction.
    So, $C$ is a bad 11-cycle. By Lemma \ref{bad cycle}, $C$ must contain $t_2$ inside and have a $(3,7,7)$-claw. Now, $C$ is adjacent to two triangles in $G$, contradicting Lemma \ref{bad cycle 1}. Therefore, the item $(b)$ of Lemma \ref{operation} holds for $T$.

    If both $t_1$ and $t_3$ are external, then $P$ is a splitting 4-path of $D$, which together with $D$ forms a 5- or 7-cycle $C$ by Lemma \ref{splitting path}. Then $T$ creates a $2$- or $4$-cycle, contradicting the truth of the item $(b)$. Hence, the item $(a)$ of Lemma \ref{operation} holds for $T$.

     Hence, $\phi_0$ can super-extend to $G'$ by Lemma \ref{operation} and further to $G$ as follows. Firstly, assume that all the vertices of $\{t_1,t_2,\ldots, t_5\}$ are of color $F$. If both the pairs $(t_1,t_2)$ and $(t_2,t_3)$ are $F$-linked, then $t_1t_3$ is contained in an $F$-cycle or a splitting $F$-path of $D$ in $G'$, a contradiction. Hence, at least one of the pairs $(t_1,t_2)$ and $(t_2,t_3)$ is not $F$-linked, say $(t_p,t_{p+1})$. Color the vertices of $f$ with $F, F, I, F, I$ in cyclic order starting from $t_p$ and then $t_{p+1},\cdots$, we are done.
     It remains to assume that there is a vertex from $\{t_1,t_2,\ldots, t_5\}$ of color $I$, say $t_q$. $I$-nicely color the path $v_{q+1}v_{q+2}\ldots v_{q-1}$, where the addition for the index runs modulo $k$. Finally, assign $v_q$ with color $F$, which 
     obviously brings neither $F$-cycles nor splitting $F$-paths of $D$.

   {\bf  Case 2}:
    Let $k=7$.
    Apply on $G$ the following operation $T$: remove all the vertice of $f$ and insert an edge between $t_1$ and $t_4$, obtaining a smaller plane graph $G'$.

    Suppose that $T$ creates a $6^-$-cycle or an ext-triangle 7- or 8-cycle. Then  $G-\{v_5,v_6,v_7\}$ has a $12^-$-cycle $C$ containing the path $P=t_1v_1v_2v_3v_4t_4$ and additionally, $\ext[C]$ has a triangle sharing an edge with $C-E(P)$ when $|C|\in \{11,12\}$.
    If $C$ is a good cycle, then $t_2,t_3\in V(C)$. Since $|C|\leq 12$, each edge of $v_1v_2v_3v_4$ is incident with a $5$-face. Now $|C|=11$, which implies that one of those $5$-faces is adjacent to a triangle, a contradiction.
    So, $C$ is a bad cycle. On one hand, $C$ has a (5,5,7)-claw or (5,5,5,7)-biclaw by Lemma \ref{bad cycle 1}. On the other hand, either $v_5,v_6,v_7\in \intt(C)$ or $C$ contains $t_2t_3$ inside by planarity. A contradiction follows. So, the item $(b)$ of Lemma \ref{operation} holds for $T$.
    
    If both $t_1$ and $t_4$ are external vertices, then $P$ is a splitting 5-path of $D$, which together with $D$ forms a $9^-$-cycle by Lemma \ref{splitting path}. Then $T$ creates a $5^-$-cycle, contradicting the truth of the item $(b)$.  So, the item $(a)$ of Lemma \ref{operation} holds for $T$.    

     Hence, $\phi_0$ can super-extend to $G'$ by Lemma \ref{operation} and further to $G$ in a similar way as for Case (1). 
\end{proof}

A \emph{3-7-face} $H$ consists of a 3-face $[xzy]$ and a 7-face $[xzv_1\ldots v_5]$ such that their common part is the edge $xz$, $z$ is an internal 4-vertex, and all other vertices of $H$ are internal 3-vertices.

\begin{figure}[ht]
	\centering
	\includegraphics[width=4cm]{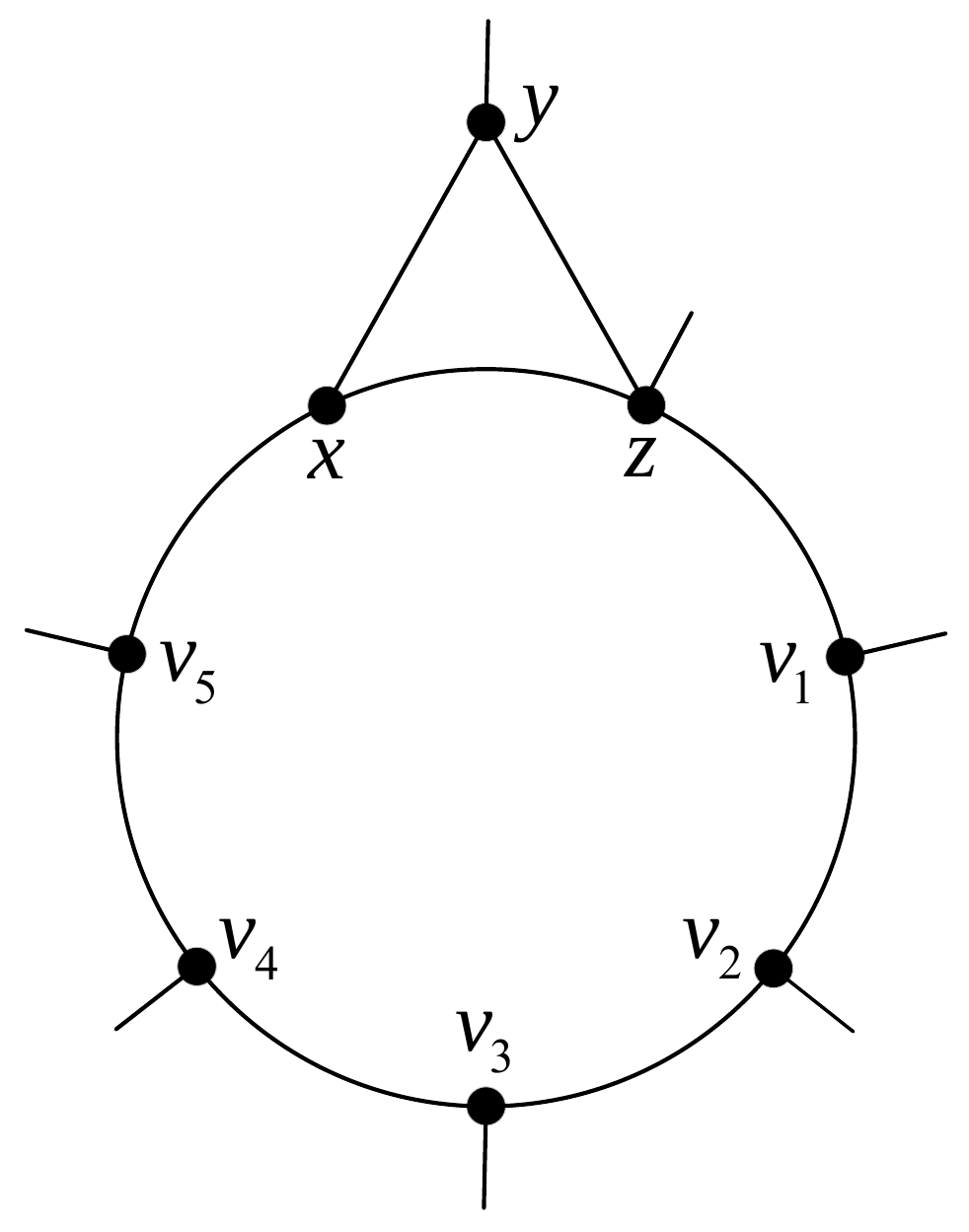}
	\caption{3-7-face}\label{fig_3-7-face}
\end{figure}

\begin{lemma} \label{lem_8-cycle}
	$G$ has no 3-7-faces.
\end{lemma}

\begin{proof}
	
	Suppose to the contrary that $G$ has a 3-7-face $H$, using the same label for vertices as in Figure \ref{fig_3-7-face}.
	The pre-coloring $\phi_0$ can super-extend to $G-V(H)$ by the minimality of $G$ and further to $G$ as follows. 
	
	$I$-nicely color the path $v_5v_4\ldots v_1zy$.
	If at least one of $y$ and $z$ is of color $I$, then assign $x$ with $F$, which brings neither $F$-cycle nor splitting $F$-path of $D$ except that $[xzv_1v_2\ldots v_5]$ might be an $F$-cycle. For this exceptional case, the remaining neighbor of each vertex from $\{z,v_1,v_2,\ldots,v_5\}$ is of color $I$. Reassign $x$ with $I$ and $y$ with $F$, which brings neither $F$-cycle nor splitting $F$-path of $D$, we are done. Hence, we may next assume that both $y$ and $z$ are of color $F$. 
	
	If $v_5$ is of color $F$, then assign $x$ with $I$, we are done. So, let $v_5$ be of color $I$. 
	Denote by $y'$ the remaining neighbor of $y$. If $y'$ is of color $F$, then recolor $y$ with $I$ and color $x$ with $F$, we are done.
	So, let $y'$ be of color $I$.	
	$F$-nicely recolor the path $v_5v_4\ldots v_1zy$, which yields that both $v_5$ and $y$ are of color $F$, but the color of $z$ might be changed. Finally, color $x$ different from $z$, we are done. 
\end{proof}

   A \emph{7-7-face} $H$ consists of two 7-faces $[xu_6...u_1]$ and $[xv_1...v_6]$ such that their common part is the vertex $x$, $u_1$ is adjacent to $v_1$, both $x$ and $u_1$ are internal 4-vertices, and all other vertices of $H$ are internal 3-vertices, see Figure \ref{fig-77face}.

\begin{figure}[ht]
	\centering
	\includegraphics[width=7cm]{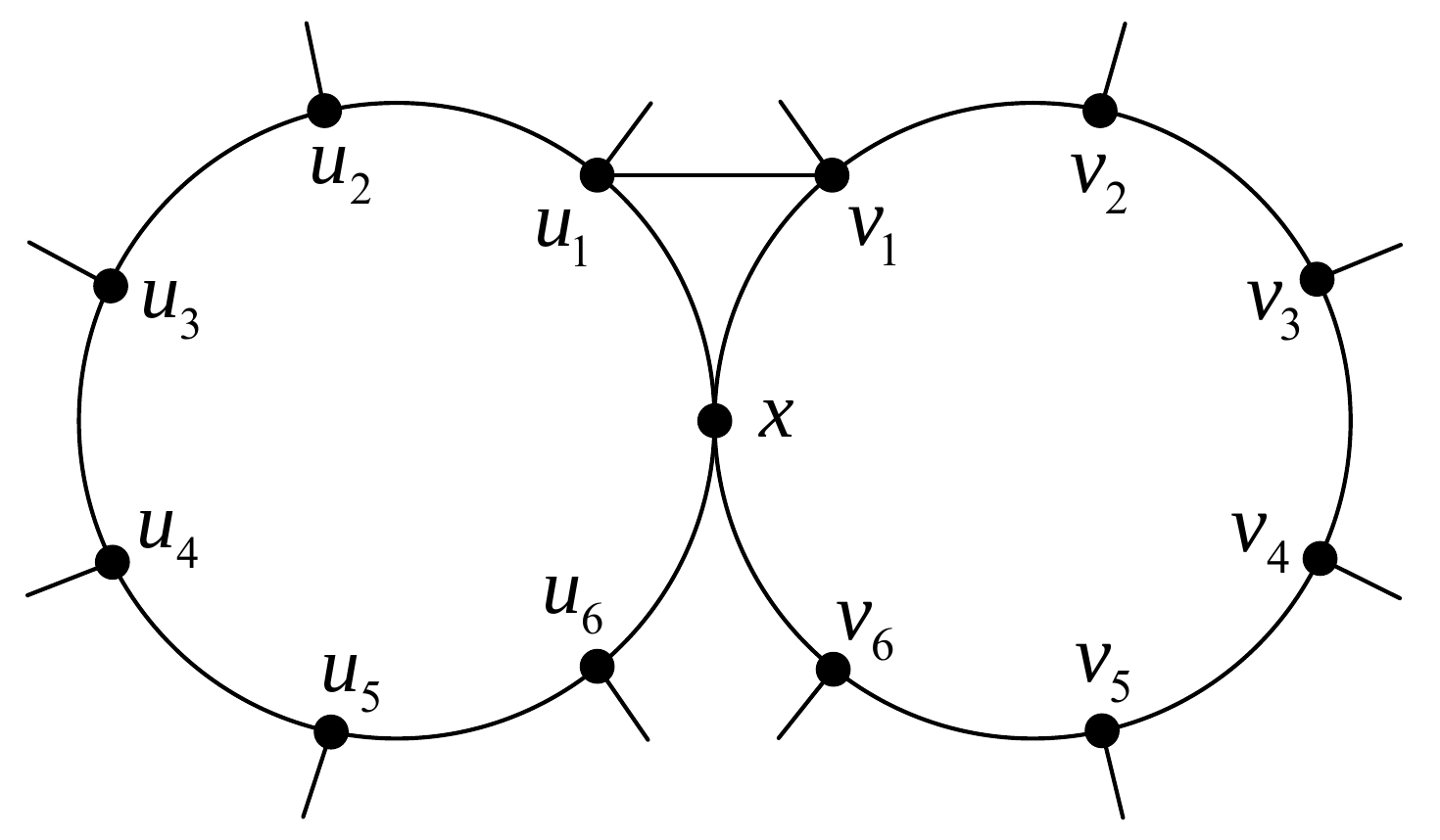}
	\caption{7-7-face}\label{fig-77face}
\end{figure}

\begin{lemma} \label{lem 7-7-face}
	$G$ has no 7-7-faces.
\end{lemma}

\begin{proof}
    Suppose to the contrary that $G$ has a 7-7-face $H$, using the same label for vertices as in Figure \ref{fig-77face}. 
    The pre-coloring $\phi_0$ can super-extend to $G-V(H)$ by the minimality of $G$ and further to $G$ as follows. Let $\vec{P_1}=u_6u_5\ldots u_1$ and $\vec{P_2}=v_6v_5\ldots v_1$.

    $I$-nicely color the path $\vec{P_1}$. If $P_1$ is an $F$-path, then $F$-nicely color the path $\vec{P_2}$. If $v_1$ is of color $F$, then assign $x$ with $I$; otherwise, reassign $v_1$ with $F$ and assign $x$ with $I$, yielding that the color of $v_1$ brings no $F$-cycle or splitting $F$-path of $D$. We are done in both cases. Hence, we assume that $P_1$ is not an $F$-path.
    
    $I$-nicely color the path $\vec{P_2}$. If $P_2$ is an $F$-path, then $u_1$ must be of color $I$. $F$-nicely recolor the path $\vec{P_1}$ regardless of the edge $u_1v_1$, yielding both $u_1$ and $u_6$ of color $F$. Assign $x$ with $I$. It is easy to see that the edge $u_1v_1$ has both ends of color $F$ and is not contained in any $F$-cycle or splitting $F$-path of $D$, we are done. Hence, we assume that $P_2$ is not an $F$-path.
    
    If not both $u_1$ and $v_1$ are of color $F$, then assign $x$ with $F$, we are done. So, assume that both $u_1$ and $v_1$ are of color $F$.
    If $v_2$ is of color $F$, then reassign $v_1$ with $I$ and assign $x$ with $F$, we are done. So, let $v_2$ be of color $I$. Denote by $t_1$ the neighbor of $u_1$ not in $H$. If $t_1$ is of color $F$, then $F$-nicely recolor the path $P_1$, yielding $u_1$ of color $I$. Assign $x$ with $F$, we are done. So, let $t_1$ be of color $I$.
    $F$-nicely recolor $P_2$, yielding $v_6,v_2,v_1$ of color $F,F,I$, repectively. Assign $x$ with $F$, which might make $u_1xv_6$ be contained in an $F$-cycle or splitting $F$-path of $D$. For this case, remove the color of $x$ and $v_1$, $F$-nicely recolor $P_1$, and assign $x$ with $I$ and $v_1$ with $F$, we are done.  
\end{proof}

A \emph{M-9-face} is a 9-face $[v_1...v_9]$ such that $v_1$,$v_2$,$v_3$,$v_5$,$v_6$,$v_7$ are six bad vertices and $v_4$ is an internal 4-vertex incident two 3-faces, see Figure \ref{fig_M9face}.

\begin{figure}[ht]
	\centering
	\includegraphics[width=6cm]{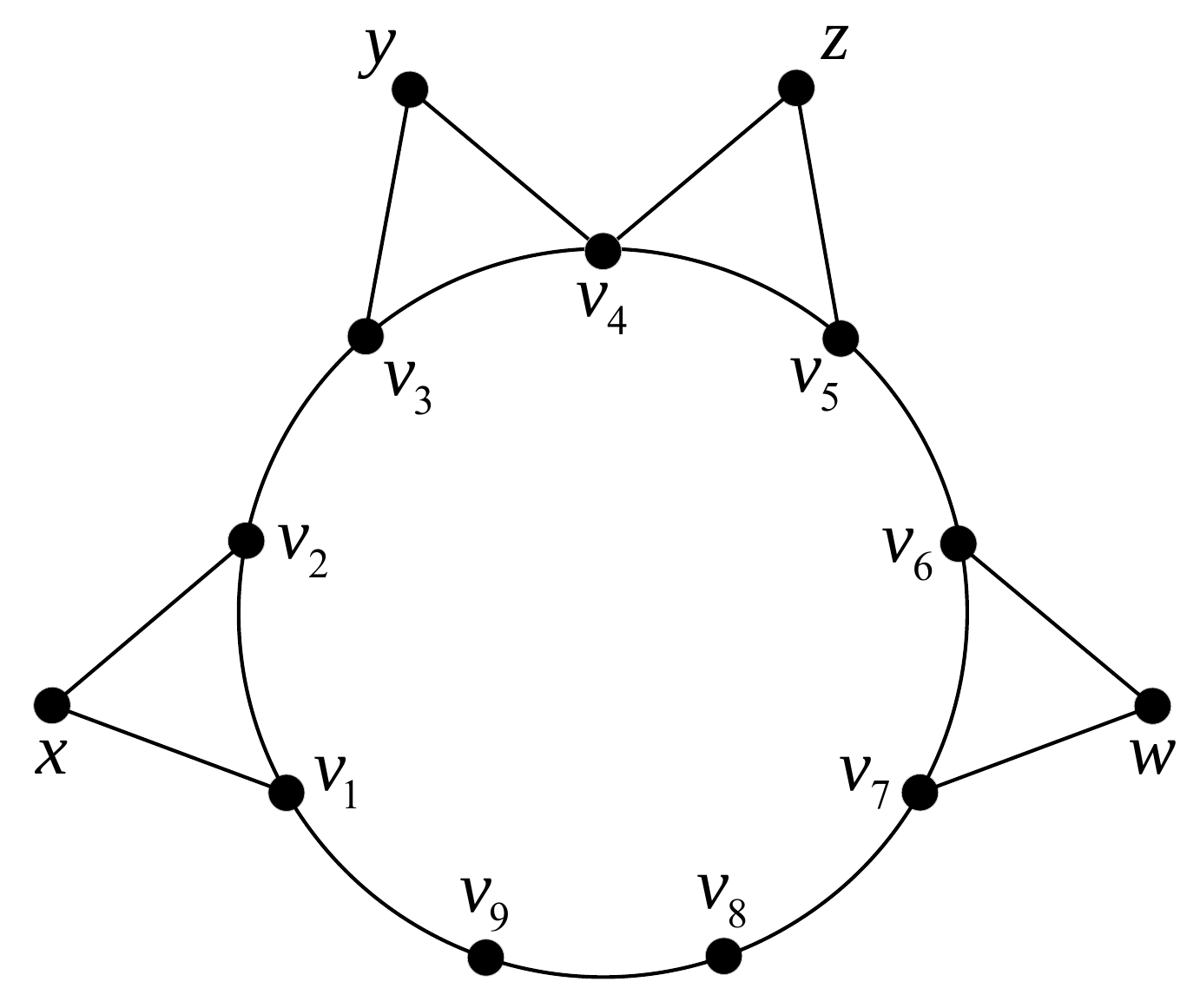}
	\caption{M-9-face}\label{fig_M9face}
\end{figure}

\begin{lemma} \label{lem_M-9-cycle}
	$G$ has no M-9-face.
\end{lemma}

\begin{proof}

    Suppose to the contrary that $G$ has an M-9-face $f$, using the same label for vertices as in Figure \ref{fig_M9face}. Let $S_1=\{v_1,v_2,v_3\}$, $S_2=\{v_5,v_6,v_7\}$, and $S=S_1\cup S_2$.
    Apply on $G$ the operation $T$ as follows: remove all the vertices of $S$ and identify two edges $v_8v_9$ and $zv_4$ so that $z$ is identified with $v_8$, obtaining a smaller plane graph $G'$. 
    Denote by $T_1$ (resp., $T_2$) the operation on $G$ consisting of remove all the vertices of $S$ and identifying $v_8$ with $z$ (resp., $v_9$ with $v_4$).
    Similarly as the proof of Lemma \ref{lem_good path}, we can deduce that both the items $(a)$ and $(b)$ hold for $T_1$ as well as $T_2$. Moreover, notice that $v_4z$ is contained in no $8^-$-cycle of $G-S$. 
    
    By Lemma \ref{operation_merge_edge}, the pre-coloring $\phi_0$ can super-extend to $G'$ and further to $G$ as follows.
    Color the vertices of $S_1$ as well as $S_2$ in the same way as we did for good path in the proof of Lemma \ref{lem_good path}. Clearly, the resulting coloring is a proper $(I,F)$-coloring of $G$. It remains to show that the coloring of $S$ brings neither $F$-cycle nor splitting $F$-path of $D$. Otherwise, denote by $H$ such a new $F$-cycle or splitting $F$-path of $D$ in $G$. 
    The way we color $S_1$ and $S_2$ implies that $V(H)\cap S_1\neq \emptyset$ and $V(H)\cap S_2\neq \emptyset$, and the coloring of $S_1$ as well as $S_2$ belongs to Case (\romannumeral2) or (\romannumeral4) of the proof of Lemma \ref{lem_good path}. Thus, all the four vertices we identified are of color $F$. So, $v_5$ is of color $I$, which yields that the coloring of $S_2$ belongs to Case (\romannumeral2) not Case (\romannumeral4) and further that the coloring of $S_2$ brings neither $F$-cycle nor splitting $F$-path of $D$, contradicting that $V(H)\cap S_2\neq \emptyset$.
\end{proof}

\subsection{Incompatibility of reducible configurations}
By exactly the same discharging procedure of the article \cite{469}, we can derive the incompatibility of reducible configurations of $G$ as depicted in Lemmas \ref{lem_min degree} to \ref{lem_M-9-cycle}. More precisely, in the subsection 2.1 of \cite{469}, the authors prove reducible configurations for minimal counterexample $H\in \mathcal{G}$, which are exactly the same as Lemmas \ref{lem_min degree} to \ref{lem_M-9-cycle} of this paper. The subsection 2.2 of \cite{469} are discharging procedure, which shows that these reducible configurations for $H$ (equivalently, Lemmas \ref{lem_min degree} to \ref{lem_M-9-cycle} for $G$) are incompatible. This incompatibility completes the proof of Theorem \ref{thm_main_extension}.

\section{Acknowledgement}
Yingli Kang is supported by National Natural Science Foundation of China (Grant No.: 11901258), Natural Science Foundation of Zhejiang Province of China (Grant No.: LY22A010016), and Department of Education of Zhejiang Province of China (Grant No.: FX2022084).
Ligang Jin is supported by National Natural Science Foundation of China (Grant No.: 11801522) and Natural Science Foundation of Zhejiang Province of China (Grant No.: LY20A010014).

\end{document}